\documentclass[12pt,onecolumn]{IEEEtran} 
\everymath{\displaystyle}
\usepackage{graphicx}
\usepackage{graphics} 
\usepackage{epsfig} 
\usepackage{mathptmx} 
\usepackage{times} 
\usepackage{amsmath} 
\usepackage{amssymb}  
\usepackage{amsthm}
\usepackage{dsfont}
\usepackage{amssymb} 
\usepackage{color}

\newcommand{\beqnum}{\begin{equation}\begin{array}{lcl}}
\newcommand{\eeqnum}{\end{array}\end{equation}}
\newcommand{\beqnom}{\begin{eqnarray}}
\newcommand{\eeqnom}{\end{eqnarray}}
\newcommand{\beqnc}{\begin{center}\begin{eqnarray}}
\newcommand{\eeqnc}{\end{eqnarray}\end{center}}
\newcommand{\beqnlm}{\begin{equation}\vspace{-.5cm}\begin{array}{lll}}
\newcommand{\eeqnlm}{\end{array}\end{equation}}\vspace{-.5cm}
\newcommand{\beq}{\begin{eqnarray*}}
\newcommand{\eeq}{\end{eqnarray*}}
\newcommand{\bef}{\begin{figure}[tbh!]}
\newcommand{\enf}{\end{figure}}

\usepackage{geometry}
\newtheorem{montheo}{\bf Theorem}
\newtheorem{rem}{\bf Remark}
\newtheorem{lemme}{\bf Lemma}

\newtheorem{Assumption}{\bf Assumption}

\title{Global tracking for an underactuated ships with bounded feedback controllers}
\author{Salah Laghrouche$^{*}$, Mohamed Harmouche$^{*}$,  and Yacine Chitour$^{**}$\\
$^{*}$SET Laboratory, UTBM, Belfort, France.\\
$^{**}$L2S, Universite Paris XI, CNRS 91192 Gif-sur-Yvette, France.\\
\small{salah.laghrouche@utbm.fr, mohamed.harmouche@utbm.fr, yacine.chitour@lss.supelec.fr }\\
}

\date{}

\begin{document}
\maketitle
\begin{abstract}
\noindent In this paper, we present a global state feedback tracking controller for underactuated surface marine vessels. This controller is based on saturated control inputs and, under an assumption on the reference trajectory, the closed-loop system is  globally asymptotically stable (GAS). It has been designed using a 3 Degree of Freedom benchmark vessel model used in marine engineering. The main feature of our controller is the boundedness of the control inputs, which is an essential consideration in real life. In absence of velocity measurements, the controller works and remains stable with observers and can be used as an output feedback controller. Simulation results demonstrate the effectiveness of this method.
\end{abstract}


\everymath{\displaystyle}

\section{Introduction}
Precise tracking control of surface marine vessels (ships and boats) is often required in critical operations such as support around off-shore oil rigs \cite{PE}. This problem is of particular interest as marine vessels are often underactuated, i.e. the number of independent actuators is less than the degrees of freedom (DOF) to be controlled. In this paper, we consider the problem of tracking control of a 3-DOF vessel model (surge, sway and yaw \cite{Fossen_book}), working under two independent actuators capable of generating surge force and yaw moment only. It has been shown in \cite{B2,CR,Z} that under Brockett's necessary condition \cite{B2}, stabilization of this system is impossible with continuous or discontinuous time-invariant state feedback. This can be seen in \cite{G} where the authors developed a continuous time-invariant controller that achieved global exponential position tracking but the vessel orientation could not be controlled. In addition, it is shown in \cite{R} that the underactuated ship can not be transformed into a driftless chained system; which means that the control techniques used for the similar problem of nonholonomic mobile robot control cannot be applied directly to the underactuated ship control. Accordingly, control of underactuated vessels in this configuration has been studied rigorously by contemporary researchers, examples of which are \cite{PN,FGBL,BOF,S1,B1}.

In \cite{R}, the author showed that under discontinuous time-varying feedback, the underactuated vessel is strongly accessible and small-time locally controllable at any equilibrium. A discontinuous time-invariant controller was proposed which showed exponential convergence of the vessel towards a desired equilibrium point, under certain hypotheses imposed on the initial conditions. In \cite{PE}, a continuous periodic time-varying feedback controller was presented that locally exponentially stabilizes the system on the desired equilibrium point by using a global coordinate transformation to render the vessel's model homogenous. In \cite{PN}, a combined integrator backstepping and averaging approach was used for tracking control, together with the continuous time-varying feedback controller for position and orientation control. This combined approach, later on used in \cite{Pettersen_IJC}, provides practical global exponential stability as the vessel converges to a neighborhoo
 d of the desired location or trajectory, the size of which can be chosen arbitrarily small. Jiang \cite{Jiang_2002} used Lyapunov's direct method for global tracking under the assumption that the reference yaw velocity requires persistent excitation condition; therefore implying that a straight line trajectory could not be tracked. This drawback was overcome in \cite{Do2002} and \cite{Behal_TAC}. Do et al. \cite{Do2002} proposed a Lyapunov based method and backstepping technique for stabilization and tracking of underactuated vessel.  In this work, conditions were imposed on the trajectory to transform the tracking problem into dynamic positioning, circular path tracking, straight line tracking and parking. Borhaug et al. have proposed a control scheme for straight line following of a formation of marine vessels in \cite{Panteley2011}. 

In this paper, we address the global tracking control of underactuated vehicles, using saturated state feedback control \cite{Chitour1995,Liu1996,Yakoubi2006,Yakoubi2007,Yakoubi2007-1,Yakoubi2007-2}. Our work addresses the remaining case not treated in
\cite{Do2002}, i.e., the yaw angle of the tracked trajectory does not admit a limit at time goes to infinity. This research is therefore in the same direction as in \cite{Chwa_2011}, where the author achieved practical stability. Our algorithm provides asymptotic convergence to the tracked trajectory from any initial point. The advantage of using saturated controls is that the global asymptotic stability is ensured while the control inputs remain bounded, as real life actuators are all limited in output, see for instance \cite{GMM2011},  \cite{GMM2012}. The proposed controller has been proven to work with state measurements, as well as with observers in the case where all states may not be measured (cf. also \cite{GMM2011}).

The paper is organized as follows: the vessel model is presented in Section 2 and the control problem is formulated in Section 3. In Section 4, the controller is developed and the proof of stability is given. In Section 5, the stability of the controller is shown in presence of observation errors. Simulations are given in Section 6 and concluding remarks are presented in Section 7.
\section{Vessel Model}
In this section, we discuss the physical model of the marine vessel and the related assumptions on physical phenomena associated with its motion. Then, a mathematical reformulation is presented, following variable and time-scale changes, to obtain a suitable form for control design.
\subsection{Physical Model}
\noindent The general 6-DOF rigid body model for surface marine vessels presented in \cite{Fossen_book} can be reduced by considering surge, sway and yaw motions only, under the following assumptions \cite{Chwa_2011},

\begin{description}
\item[$(H1)$] Heave, roll and pitch motions induced by drift forces of wind, wave and ocean current are neglected.
\item[$(H2)$] The inertia, added mass and hydrodynamic damping matrices are diagonal. 
\end{description}

\noindent The aft propeller configuration provides only the surge force $\tau_u$ and the yaw moment $\tau_r$. The kinematic and dynamic equations of the vessel can therefore be written as \cite{Do2002,Chwa_2011}


\beqnum \label{system}
\left\{ \begin{array}{lll}
\dot x &=& u \cos \left( \psi  \right) - v\sin \left( \psi \right),\quad
\dot y = u \sin \left( \psi  \right) + v\cos \left( \psi \right),\\
\dot u &=& \frac{1}{c}vr - au + \bar \tau_1 ,\quad
\dot v =  - cur - bv ,\\
\dot \psi  &=& r ,\quad \dot r = \kappa uv - dr + {\bar \tau _2},
\end{array} \right.
\eeqnum
where $(x,y)$ and $\psi$ are the coordinates and the yaw angle of the vessel in the earth-fixed frame, and $u$, $v$ and $r$ denote the surge, sway and yaw velocities respectively. The control inputs $\bar \tau_1$ and $\bar \tau_2$ are the normalized expressions of the surge force and yaw moment, given as
	$\bar \tau_1 = \frac{1}{m_1}{\tau _u}$ and ${\bar\tau _2} = \frac{1}{m_3}{\tau_r}$.
The parameters $a$, $b$, $c$, $d$ and $\kappa$ are positive constants that represent the mechanical properties of the system, namely the inertia $m_i > 0$ and  hydrodynamic damping $d_i$, where $i=1,\ 2,\ 3$ corresponds to surge, sway and yaw motions respectively. The constants are defined as follows $a = \frac{d_1}{m_1}$, $b = \frac{d_2}{m_2}$, $c = \frac{m_1}{m_2}$, $ d = \frac{d_3}{m_3}$, $\kappa = \frac{m_1 - m_2}{m_3}$.


\subsection{Model for control}
For control design, the system model \eqref{system} can be simplified by normalizing the physical parameters through straightforward variable and time-scale changes. For the sake of clarity, let us rewrite System \eqref{system} as follows,
\beqnum \label{system1}
(\bar S)\left\{ \begin{array}{cll}
{\left( \begin{array}{l}
		\dot x\\
		\dot y
\end{array} \right) } &=& {R_\psi }\left( \begin{array}{l}
		u\\
		v
\end{array} \right),\\
{\left( \begin{array}{l}
		\dot u\\
		\dot v
\end{array} \right)} &=&  - D_0 \left( \begin{array}{l}
		u\\
		v
\end{array} \right) - r A_c \left( \begin{array}{l}
		u\\
		v
\end{array} \right) + \left( \begin{array}{l}
		1\\
		0
\end{array} \right){\bar \tau_1},\\
\dot \psi  &=& r ,\quad \dot r = \kappa uv - dr + \bar \tau _2 ,
\end{array} \right.
\eeqnum
where the matrices $D_0$,  $R_\psi$ and $A_c$ are given as
\beqnum
D_0 = \left( {\begin{array}{*{20}{c}}
		a & 0\\
		0 & b
\end{array}} \right),
\ {R_\psi } = \left( {\begin{array}{*{20}{c}}
		{\cos \left( \psi  \right)} & { - \sin \left( \psi  \right)}\\
		{\sin \left( \psi  \right)} & {\cos \left( \psi  \right)}
\end{array}} \right),\
A_c = \left( \begin{array}{cc} 0 & {-1/c}\\ c & 0  \end{array}  \right).
\eeqnum
Let us consider the following matrix $D_{\rho}=diag(\rho,c\rho)$, 
where $\rho$ is a positive constant to be chosen later. Then we obtain
$A_1 = D_{\rho}^{-1} A_c D_{\rho}$. The time scale $s:=dt$ is introduced in System \eqref{system1} as well as the linear changes of variables $(u(t)/d\rho,v(t)/dc\rho)$ and $r(t)/d$, still denoted $(u(s),v(s)$ and $r(s)$ respectively.


After easy computations and by setting $\beta: = \frac{\kappa}{c\rho^2}$, $\tau_1: = \frac{\bar \tau_1}{ \rho d^2}$, $\tau_2 := \frac{\bar \tau_2}{ d^2}$ and $D=diag(a/d,b/d)$,
the dynamics of the vessel, denoted by  $(S)$, is rewritten as follows,
\beqnum\label{system2_timescale}
(S)\left\{ \begin{array}{cll}
{\left( \begin{array}{l}
		\dot x\\
		\dot y
\end{array} \right) } &=& {R_\psi} D_\rho \left( \begin{array}{l}
		u\\
		v
\end{array} \right),\\
{\left( \begin{array}{l}
		\dot u\\
		\dot v
\end{array} \right) } &=&  - D \left( \begin{array}{l}
		u\\
		v
\end{array} \right) - r A_1 \left( \begin{array}{l}
		u\\
		v
\end{array} \right) + {\tau_1} \left( \begin{array}{l}
		1\\
		0
\end{array} \right) ,\\
\dot \psi  &=& r,\quad
\dot r = \beta uv - r + \tau_2.
\end{array} \right.
\eeqnum

\section{Problem Formulation}
The goal of this paper is tracking control of the presented underactuated marine vessel by controlling its position and orientation. The vessel is forced to follow a reference trajectory which is generated by a ``virtual vessel'', as follows \cite{Lapierre2003,Burger2010},
\beqnum \label{reference}
(S_{re}) \left\{ \begin{array}{cll}
{\left( \begin{array}{l}
	\dot x_{re}\\
	\dot y_{re}
\end{array} \right) } &=& R_{\psi_{re}} D_\rho \left( \begin{array}{l}
	u_{re}\\
	v_{re}
\end{array} \right),\\
{\left( \begin{array}{l}
	\dot u_{re}\\
	\dot v_{re}
\end{array} \right) } &=&  - {D}\left( \begin{array}{l}
	u_{re}\\
	v_{re}
\end{array} \right) - r_{re} A_1 \left( \begin{array}{l}
	u_{re}\\
	v_{re}
\end{array} \right) + \left( \begin{array}{l}
	1\\
	0
\end{array} \right) \tau_{1,re},\\
\dot \psi_{re} &=& r_{re},\quad
\dot r_{re} = \beta u_{re} v_{re} - r_{re} + \tau_{2,re},
\end{array} \right.
\eeqnum
where all variables have similar meanings as in System \eqref{system2_timescale}. Tracking control is achieved by using saturated control inputs and under the assumption that the velocities are bounded \cite{chitour_2001,ACTA2012}. This assumption holds true physically as resistive drag forces increase as the velocity increases and therefore the latter cannot increase indefinitely if the control is bounded. These assumptions are also valid for the reference system and are formalized in the following manner:  
\begin{Assumption} \label{ass1}
\noindent There exist constraints on the control inputs and velocities such that
\beqnum
	\left| \bar \tau _1 \right| \le \bar \tau_{1,\max},\ \left| \bar \tau_2 \right| \le \bar \tau_{2,\max},\
	\left| u \right| \le \bar u_{\max},\ \left| v \right| \le \bar v_{\max},
\eeqnum
where $\bar \tau_{1,\max}$, $\bar \tau_{2,\max}$, $\bar u_{\max}$ and $\bar v_{\max}$  are known positive constants. The velocities $u_{re}$, $v_{re}$ and the forces $\tau_{1,re}$ and $\tau_{2,re}$ verify the same bounds as above and the reference angle $\psi_{re}$ does not converge to a finite limit as $t$ tends towards infinity.
\end{Assumption}

The variable and time-scale change defined in the previous section requires the following new bounds to be defined for the new control inputs $\tau_1$ and $\tau_2$, denoted by $\tau_{1,max}$ and $\tau_{2,max}$ respectively: $\tau_{1,max} = \frac{\bar \tau_{1,max}}{\rho d^2}$ and $\tau_{2,max} = \frac{\bar \tau_{2,max}}{d^2}$.
We consider the following condition upon the saturation limits of the control inputs, to be used later on in the control design. We use here $m_1$ to denote $\min(a_1/2,b_1)$.
\beqnum\label{C1}
\text{\textbf{C1: }}\quad  \beta \frac{\tau_{1,max}^2}{a_1 m_1} < \tau_{2,max}.
\eeqnum
Note that this condition is always satisfied by choosing $\rho > \frac{\bar \tau_{1,max}}{d}\sqrt{\frac{\beta}{a_1 m_1 \bar \tau_{2,max}}}$. Our control objective is that $(S)$ follows $(S_{re})$. With respect to the frame of reference of the reference trajectory $(S_{re})$, the error system is defined as
\beqnum \label{error1}
e_x=\cos(\psi_{re})(x - x_{re})+\sin(\psi_{re})(y - y_{re}), \quad 
e_y=-\sin(\psi_{re})(x - x_{re})+\cos(\psi_{re})(y - y_{re}),\\
e_u=u- u_{re},\ e_v=v-v_{re},\ e_\psi=\psi-\psi_{re},\ e_r=r-r_{re}.
\eeqnum
Defining new controllers $w_1$ and $\tilde w_2$, as follows, $w_1: = \tau_1 - \tau_{1,re}$ and $\tilde w_2: = \tau_2 - \tau_{2,re}$,
the dynamics of system \eqref{error1} becomes
\beqnum \label{sys_error}
(S_e)\left\{ \begin{array}{cll}
{\left( \begin{array}{l}
	\dot e_x\\
	\dot e_y
\end{array} \right) } &=&  - r_{re} A_1 \left( \begin{array}{l}
	e_x\\
	e_y
\end{array} \right) + D_\rho \left( \begin{array}{l}
	e_u\\
	e_v
\end{array} \right) + sin(e_\psi) A_1 D_\rho \left( \begin{array}{l}
	u_{re}\\
	v_{re}
\end{array} \right)\\
	&& + \left( cos\left( e_\psi \right) - 1 \right) D_\rho \left( \begin{array}{l} u_{re} \\ v_{re} \end{array} \right)
		 + \left( R_{e_\psi} - Id_2 \right) \left( \begin{array}{l} e_u \\ e_v  \end{array} \right),\\
		
{\left( \begin{array}{l}
	\dot e_u\\
	\dot e_v
\end{array} \right) } &=&  - D \left( \begin{array}{l}
	e_u\\
	e_v
\end{array} \right) - r_{re} A_1 \left( \begin{array}{l}
	e_u \\
	e_v
\end{array} \right) - e_{r} A_1 \left( \begin{array}{l}
	u_{re}\\
	v_{re}
\end{array} \right) + \left( \begin{array}{l}
	1\\
	0
\end{array} \right) w_1 +   e_r  \left( \begin{array}{c} - e_v\\ e_u  \end{array} \right),\\
\dot e_\psi &=& e_r,\quad
\dot e_r = \beta \left( uv - u_{re} v_{re} \right) - e_r + \tilde w_2.\end{array} \right.
\eeqnum
The control objective is to force the error system $(S_e)$ to $0$, using $w_1$ and $\tilde w_2$.

\section{Controller design}
We first develop the following intermediate result, concerning the bounds of $u$, $v$, $r$. 
\begin{lemme}\label{limsup_uvr}
The variables $u$, $v$, $r$ are bounded and satisfy
\beqnum \label{ineq_uvr}
\begin{array}{cll}
	\mathop{\lim \sup}\limits_{t\to\infty}  \Vert  (u,v) \Vert &\le& \frac{\tau_{1,\max }}{2\sqrt{a_1m_1}},\quad
	\mathop{\lim \sup}\limits_{t\to\infty} \left| r \right| \le \tau_{2,\max} + \beta \frac{\tau_{1,\max}^2}{2a_1 m_1}.
\end{array}
\eeqnum
\end{lemme}
\begin{proof}
Let us consider the positive definite function
\beq V_{u,v} = \frac{1}{2}(u^2 + v^2) ,\eeq then
\beq
\dot V_{u,v} = u\dot u + v\dot v  
 &=&u\left( - a_1 u + rv  + \tau_1 \right) + v \left( { - b_1 v - r u} \right)\\
 &\leq&   -\frac{a_1}2u^2- b_1 v^2 + \frac{\tau_1^2}{2 a_1}\\  
 &\leq&  - m_1(u^2+v^2) + \frac{\tau_1^2}{2 a_1},
\eeq
we can see that $\dot V_{u,v} \le 0$ for $(u^2 + v^2)\ge \frac{\tau_1^2}{2 a_1 m_1}$, which directly implies the first inequality in \eqref{ineq_uvr} (see Chapter 9 of \cite{Khalil}). Similarly, consider $V_r = \frac1{2}r^2$ then,
$
\dot V_r = r\dot r = -r^2 - r \left(\tau_2 + \beta uv \right)=-r(r+\tau_2 + \beta uv ),
$
$\dot V_r\le 0 $ for $|r| \ge |\tau_2|+|\beta u v|$, from which we derive the second inequality in \eqref{ineq_uvr}.
\end{proof}

\begin{rem}\label{limsup_uvr_ref}
As the reference trajectory system is similar to the vessel model, it can be shown that the limits defined in Lemma \ref{limsup_uvr} are valid for $u_{re}$, $v_{re}$, $r_{re}$ as well.
\end{rem}

\noindent We define a new control variable $w_2: = \beta \left( uv - u_{re} v_{ref} \right) + \tilde w_2$. As the upper bounds of $u$, $v$, $u_{re}$ and $v_{ref}$ are known according to Lemma \ref{limsup_uvr} and Remark \ref{limsup_uvr_ref}, we obtain
	\beqnum
		\mathop{\lim \sup}\limits_{t\to\infty} \beta \left| uv - u_{re} v_{ref}  \right| \le \beta \frac{\tau_{1,max}^2}{a_1 m_1}.
	\eeqnum
If $\vert w _2\vert\leq U_2$, for a positive constant $U_2$, one must have $U_2 + \beta \frac{\tau_{1,max} ^ 2}{a_1 m_1} \le \tau_{2,max}$, which is guaranteed by Condition \textbf{C1}.
With these preliminaries established, we will now proceed to fulfill the control objective by using the bounded controls $w_1$ and $w_2$. Let $\sigma (.)$ be a standard saturation function,
i.e., $\sigma(t)=\frac{t}{\max(1,\vert t\vert)}$.  The main result of the paper is given next.

\begin{montheo}\label{main_theo}
If Assumption \ref{ass1} and Condition \textbf{C1} are fulfilled, then for an appropriate choice of constants $U_1,\ \rho,\ \xi,\ M,\ U_2,\ k_1,\ k_2,\ \mu$ satisfying:
\beq
	a_1 > U_1 + \rho,\  U_1 > \frac{\left| a_1 - \frac{b_1}{c} \right|}{ \min\left( a_1 ,\  \frac{b_1}{c} \right) }\rho,\ U_2>0,\\
	M>0,\  k_1>k_2-1>0,\ 	a_1+\xi= \mu \rho,\  b_1 = \mu c \rho,
\eeq
then the following controller ensures global asymptotic stability of the tracking error system $(S_e)$:
\beqnum\label{toto-2}
	w_1 &=&  - U_1 \sigma \left(  \frac{\xi e_u}{U_1}\right) - \rho \sigma \left(  M (e_x + \frac1{\mu} e_u)\right),
\\ w_2 &=&  - U_2 \sigma \left(  \frac{k_1}{U_2} e_\psi +\frac{k_2 - 1}{U_2} e_r \right).
\eeqnum
\end{montheo}

\noindent \textbf{Proof of Theorem~\ref{main_theo}}.

We first consider the errors $e_\psi$ and $e_r$ and take $k_1,k_2>0$ large in the control input $w_2$ defined previously.
The dynamics of $e_\psi$ and $e_r$ in $(S_e)$ $\dot e_\psi = e_r$ and $\dot e_r = - e_r - U_2 \sigma \left(  \frac{k_1}{U_2} e_\psi +\frac{k_2 - 1}{U_2} e_r \right)$.
\begin{lemme}\label{lemme2}
If $U_2 > 0$ and $k_1 > k_2-1>0$, then after a sufficiently large time, the saturated control operates in its linear region and the errors $e_\psi$ and $e_r$ converge to zero exponentially.
\end{lemme}

\begin{proof}
Let us consider the Lyapunov function $V$,
\beqnum \label{Lyapunov1}
	V = \frac{\alpha}{2} e_r^2 + S\left( \frac{k_1}{U_2} e_\psi + \frac{k_2 - 1}{U_2} e_r \right),
\eeqnum
where $S(\xi):=\int^{\xi}_{0} \sigma(s) ds$, and $\alpha = \frac{k_1 - k_2 + 1}{U_2^2}>0$.
Set $z:=\frac{k_1}{U_2} e_\psi + \frac{k_2 - 1}{U_2} e_r$. Then, one has 
\beqnum
	\dot V = \alpha e_r \dot e_r + \sigma(z)\dot z
= -\alpha e_r^2 - (k_2 - 1 ) \sigma^2(z).\eeqnum
Then $\dot V<0$ for $(e_\psi , e_r) \neq (0,0)$; and after a finite time we obtain
$\left| \frac{k_1}{U_2} e_\psi + \frac{k_2 - 1}{U_2} e_r \right| \le 1$.
The dynamics of $e_\psi$ and $\ e_r$ becomes linear, i.e., 
$\dot e_\psi=e_r$ and $\dot e_r=- k_1 e_\psi -k_2e_r$.
As $k_1,k_2>0$, $(e_\psi,e_r)$ converges exponentially to zero.
\end{proof}

Lemma \ref{lemme2} shows that the errors  $e_\psi$ and $e_r$ converge to zero under the control $w_2$. We will now consider the errors $e_u$ and $e_v$. We choose the constants $\mu$ and $\xi$ such that $a_1 + \xi =\mu\rho$ and $b_1=\mu c\rho$, implying that 
$\xi=b_1/c - a_1$ and $\mu>0$.
\begin{lemme} \label{conv_eu_ev}
Consider the dynamics of $e_u$ and $e_v$ given in Equation \eqref{sys_error}. If $U_1$ and $\rho$ are chosen as
$a_1 > U_1 + \rho$ and $U_1 > \frac{\left| a_1 - \frac{b_1}{c} \right|}{ \min\left( a_1 , \frac{b_1}{c} \right) }\rho$, 
then the control
$w_1: = -U_1 \sigma \left(\frac{\xi e_u}{U_1}\right) - \rho \sigma_1 (.)$,
with $ \sigma_1 (.)$ to be chosen later, ensures that $e_u$ and $e_v$ satisfy the following inequalities:
\beqnum\label{toto_w1}
	\mathop {\lim \sup }\limits_{t \to \infty } \Vert (e_u,e_v)  \Vert\leqslant \frac{\rho }{\sqrt{m_2\tilde a}},\hbox{ with }
\tilde a = \mathop {\inf }\limits_{t > 0} \left(a_1+\xi\frac{\sigma \left (\frac{\xi e_u}{U_1}\right)}{\frac{\xi e_u}{U_1}}\right)> 0, \quad m_2:=\min(\tilde a/2,b_1).
\eeqnum

\end{lemme}
\begin{proof} Notice that $\tilde a>0$ since it is trivially the case if $\xi\geq 0$, and otherwise, $\tilde a\geq a_1+\xi=\frac{b_1}c$. Then, one has $e_u\dot e_u+e_v\dot e_v=-a_1 e_u^2 -b_1 e_v^2+
 e_r(e_u v_{re}- e_v u_{re}) + e_u w_1$. By applying the control $w_1$, we get
$e_u\dot e_u+e_v\dot e_v\leq
-a_1 e_u^2 -b_1 e_v^2 -U_1 e_u \sigma \left(\frac{\xi e_u}{U_1}\right) - \rho e_u \sigma_1 (.) +C_0 \left|e_r\right|\sqrt{e_u^2+e_v^2}$,
where $C_0:= \bar u_{max} + \bar v_{max}$. According to Lemma \ref{lemme2}, $e_r$ tends to zero, i.e., for large $t$, one has
$e_u\dot e_u+e_v\dot v\leq
- \left[a_1+\xi s(t)  \right] e_u^2 - b_1 e_v^2 - \rho e_u \sigma_1 (.)$.
Then \eqref{toto_w1} results from the following inequality
\beqnum\label{label_lemma4}
e_u\dot e_u+e_v\dot v \leq
-m_2(e_u^2+e_v^2)+\frac{\rho^2\sigma_1^2(\cdot)}{\tilde a}.
\eeqnum
\end{proof}
Lemma \ref{conv_eu_ev} proves the convergence of $e_u$ and $e_v$ to a neighborhood of zero.  Since
$\tilde a \ge \min \left( a_1 , \frac{b_1}{c} \right)$, one gets that $\mathop {\lim \sup }\limits_{t \to \infty } \left| \frac{\xi e_u}{U_1} \right|  < 1$, and the controller exits saturation in finite time and enter its linear region of operation.  We get $\sigma\left( \frac{\xi e_u}{U_1} \right) = \frac{\xi e_u}{U_1}$,
and the dynamics of $e_u$ and $e_v$ become
\beq\label{e_uv_sat}
 \dot e_u = - \mu\rho e_u+r_{re}e_v- \rho\sigma _1(.)+e_rv_{re}-e_re_v,\quad
  \dot e_v = - \mu c\rho e_v-r_{re}e_u-e_ru_{re}+e_re_u.
  \eeq
Define $W=\left(W_1, W_2\right)^T:=(e_x,e_y)^T+(e_u,e_v)^T/\mu$.
Then one has the following result.
\begin{lemme}\label{bounded_W}
Let $W$ defined previously and the controller $w_1$ given in \eqref{toto-2} with $\sigma_1(.) = \sigma(M W_1)$, where $M$ is an arbitrary positive constant. Then
$W$ tends to a finite limit $\bar W =(0, \bar W_2 )^T$.
\end{lemme}
\begin{proof}
The dynamics of $W$ can be expressed as
\beq\label{dynamic_W}
\dot W &=&  - {r_{re}}{A_1}W + \frac{r_{re}}\mu{A_1}\left( {\begin{array}{c}
  {e_u} \\
  {e_v}
\end{array}} \right) + \sin \left( {{e_\psi }} \right){A_1}{D_\rho }\left( {\begin{array}{c}
  {u_{re}} \\
  {v_{re}}
\end{array}} \right) + {D_\rho }\left( {\begin{array}{{c}}
  {e_u} \\
  {e_v}
\end{array}} \right)+ O\left( {e_\psi ^2,\left| {e_\psi } \right|\left( {\begin{array}{c}
  {e_u} \\
  {e_v}
\end{array}} \right)} \right)\\ &&- {D_\rho }\left( {\begin{array}{c}
  {e_u} \\
  {e_v}
\end{array}} \right) - {r_{re}}\frac{A_1}{\mu }\left( {\begin{array}{c}
  {e_u} \\
  {e_v}
\end{array}} \right) - \frac{{{\rho }}}{\mu }\sigma(M W_1)\left( {\begin{array}{{c}}
  1 \\
  0
\end{array}} \right)+ O\left( {\left| {e_r} \right|\left( {1 + \left\| (e_u,e_v) \right\|} \right)} \right).
\eeq

\noindent In order to find the limsup of $\|W\|$, we calculate
\beq
{W^T}\dot W &=& \sin \left( {{e_\psi }} \right){W^T}{A_1}{D_\rho }\left( {\begin{array}{*{20}{c}}
  {{u_{re}}} \\
  {{v_{re}}}
\end{array}} \right) + {W^T}\left[ {O\left( {e_\psi ^2,\left| {{e_\psi }} \right|\left( {\begin{array}{*{20}{c}}
  {{e_u}} \\
  {{e_v}}
\end{array}} \right) + O\left( {\left| {{e_r}} \right|} \right)} \right)} \right] - \frac{{{\rho }}}{\mu }{W_1}\sigma(M W_1) , \\
&=&O\left( {\left\| W \right\|.\left\| {{e_\psi },{e_r}} \right\|} \right) - \frac{{{\rho }}}{\mu }{W_1}\sigma(M W_1),
\eeq
which implies that $|W^T \dot W| +\frac{\rho}{\mu}W_1\sigma_1(MW_1)\leqslant \|W\| O\left( \|e_\psi,e_r\|\right)$. 
One deduces first that the time derivative of $\|W \|$ is integrable over 
$\mathbb{R}_+$ and thus $W$ admits a limit as $t$ tends to infinity. Therefore, the right-hand side of the previous inequality is integrable over $\mathbb{R}_+$ implying the same conclusion for $W_1 \sigma(M W_1)$. As both $W_1$ and $\dot W_1$ are bounded, then according to Barbalat's Lemma, $W_1\rightarrow 0$ as $t\rightarrow \infty$. Consequently $W_2$ tends towards a finite value $\bar W_2$ as $t$ tends to infinity.
\end{proof}
Lemma \ref{bounded_W} permits us to further improve the result of Lemma \ref{conv_eu_ev}, as follows.
\begin{lemme}\label{lemme5}
If Lemmas \ref{conv_eu_ev} and \ref{bounded_W} hold true, then $e_u$ and $e_v$ converge to zero asymptotically.
\end{lemme}
\begin{proof}
From Lemma \ref{conv_eu_ev} and setting $G(e_u,e_v):=(e_u^2+e_v^2)/2$, Equation \eqref{label_lemma4} is rewritten as
$\dot G +2m_2 G\leq \frac{\rho^2\sigma_1^2(MW_1)}{\tilde a}$.
One concludes using Barbalat's Lemma. 
\end{proof}
So far, we have established that the errors $e_\psi$, $e_r$, $e_u$ and $e_v$ converge to zero. From Lemmas \ref{bounded_W} and \ref{lemme5}, we deduce that if $W_1 \rightarrow 0$ and $e_u \rightarrow 0$, then $e_x $ will converge asymptotically to zero as well. We next address the convergence of the remaining error variable, $e_y$.
\begin{lemme}\label{lemmeW}
If Assumption \ref{ass1} is satisfied, then $\bar W_2=0$ and $e_y$ converges asymptotically to zero.
\end{lemme}
\begin{proof}
From Equation \eqref{dynamic_W} in Lemma \ref{bounded_W}, the dynamics of $W$ can be expressed as follows, $\dot W =  - {r_{re}}{A_1}W + O\left( {\left| {{e_\psi }} \right|,\left| {{e_r}} \right|,{W_1}\sigma \left( {M{W_1}} \right)} \right)$.
We define the new variable $\tilde W$ as follows, $\tilde W:=R_{\psi _{re}}W$,
and the dynamics of $\tilde W$ is given by
\beq
\dot {\tilde W}= {{\dot \psi }_{re}}{A_1}{R_{{\psi _{re}}}}W - {r_{re}}{R_{{\psi _{re}}}}{A_1}W + {R_{{\psi _{re}}}}O\left( {\left| {{e_\psi }} \right|,\left| {{e_r}} \right|,{W_1}\sigma \left( {M{W_1}} \right)} \right)=O\left( {\left| {{e_\psi }} \right|,\left| {{e_r}} \right|,{W_1}\sigma \left( {M{W_1}} \right)} \right).
%
\eeq
Since  $e_\psi$ and $e_r$ converge exponentially  to zero and $W_1 \sigma(M W_1)$ is integrable over $\mathbb{R}_+$, then $\left\| \dot \tilde W\right\|$ is also integrable over $\mathbb{R}_+$, which means that $\tilde W$ converges to a finite limit $W^l$. Then, one gets that $-W_2\sin \left( \psi _{re} \right)$ and $W_2\cos \left( \psi _{re} \right)$ tend to 
$W^l_1$ and $W^l_2$ respectively, as $t$ tends to infinity. If $\bar W_2 \ne 0$, then one easily shows that $\psi_{re}$ converges to a finite limit by considering whether $\left( W_{re}\right)_2 = 0 $ or not. That contradicts Assumption 2 and $W$ must converge asymptotically to as well as $e_y$.
\end{proof}
It should be noted that the controller presented in Theorem 1 has been designed under the assumption that all state variables are known. In the next section, the study is extended to the case where only the position and orientation states of the vessel are available and the velocities need to be observed.
\section{Tracking without velocity measurement}
\noindent In practical cases, only position and orientation feedbacks are available for navigation. Therefore the only available states of the vessel are $x,\ y,\ \psi$ along with the the complete coordinate state set of the virtual vessel to be followed. For such output feedback systems, the variables $u,v,r$ need to be observed. In this section, we will show that the controller presented in Theorem \ref{main_theo} is applicable in this case and the use of observation instead of measurement does not affect the stability.
We suppose that the velocities are obtained through an observer such as that presented by Fossen and Strand in \cite{Fossen_observer}, or a robust differentiator such as that presented in \cite{chitour_diff}. In both cases, observation errors converge exponentially to zero. It can be noted that, when we use a differentiator, the estimated values $(\hat u,\hat v, \hat r )$ of $(u,v,r)$, can be determined according to the following equation $( \hat u,\hat v)^T=D_\rho^{-1} R_{-\psi}(\hat {\dot x},\hat {\dot y} )^T$ and $\hat r =  \hat{\dot \psi}$,
where $(\hat {\dot x}, \hat {\dot y}, \hat {\dot \psi})$ are the estimated values of $( {\dot x},  {\dot y} , {\dot \psi})$ respectively.

Let us follow the same steps used in the demonstration of stability of system $(S_e)$ with velocity measurement. The observation error related to the velocity are define as below:
$f_u = u - \hat u$, $f_v = v - \hat v$ and $f_r = r - \hat r$.
As the references are common, the observation errors can be described in terms of trajectory pursuit errors as $f_u = e_u - \hat e_u$, $f_v = e_v - \hat e_v$, and $f_r = e_r - \hat e_r$,
where, $\hat e_u = \hat u - u_{re}$, $\hat e_v = \hat v - v_{re}$ and $\hat e_r = \hat r - r_{re}$.
We note that the variable $x, y, \psi$ are measured and the related observation errors are null.
The problem is transformed to demonstrate the stability of the error system $(S_e)$ under control laws $w_1$ and $\tilde {w}_2$, which are now based on the observed values.\\
%
As in the previous case, we define
$w_2:=\beta \left( \hat u \hat v - u_{re} v_{re} \right) + \tilde w_2$.
Then the result of this section can be stated as the following theorem.
\begin{montheo}\label{main_theo_obs}
 If Assumption \ref{ass1} and Condition \textbf{C1} are fulfilled, then for an appropriate choice of constants $U_1,\ \rho,\ \xi,\ M,\ U_2,\ k_1,\ k_2,\ \mu$, the following controller ensures global asymptotic stability of the tracking error system $(S_e)$:
\beqnum
	w_1 =  - U_1 \sigma \left(  \frac{\xi \hat e_u}{U_1}\right) - \rho \sigma \left( M(e_x + \frac1{\mu} \hat e_u) \right),
	w_2 =  - U_2 \sigma \left(  \frac{k_1}{U_2} e_\psi +\frac{k_2 - 1}{U_2} \hat e_r \right),
\eeqnum
if in addition  the observer errors $f_u$, $f_v$ and $f_r$ converge asymptotically to zero and are integrable over $\mathbb{R}_+$ (i.e., the integrals of their norms over $\mathbb{R}_+$ are finite).
\end{montheo}
\begin{rem}
The choice of constants $U_1,\ \rho,\ \xi,\ M,\ U_2,\ k_1,\ k_2,\ \mu$ remain the same as in the case of Theorem \ref{main_theo}, therefore their expressions and conditions will not be repeated in this section.
\end{rem}
\begin{proof}
The proof of Theorem \ref{main_theo_obs} is largely based upon the proof of Theorem \ref{main_theo}, and is developed similarly. We first consider the dynamics of error variables  $e_\psi$ and $e_r$:
\beqnum \label{integ_perturbed}
		\dot e_\psi = e_r,\quad
		\dot e_r= - e_r - U_2 \sigma \left( \frac{k_1}{U_2} e_\psi + \frac{k_2 - 1}{U_2} e_r  + f_1(t)\right) + f_2(t),
\eeqnum
where $f_1(t)=-\frac{k_2 - 1}{U_2} f_r$ and $f_2(t)=\beta( uv - \hat u \hat v )$. 
\begin{lemme}\label{lemme8_erpsi}
If $f_1(t)$ and $f_2(t)$ converge to zero asymptotically, then for some large positive constants $k_1$ and $k_2$, System \eqref{integ_perturbed} is globally asymptotically stable.
\end{lemme}
\begin{proof}
Consider the Lyapunov function $V$ defined in \eqref{Lyapunov1}. If $z:=\frac{k_1}{U_2} e_\psi + \frac{k_2 - 1}{U_2} e_r$, one gets 
	 \beqnum
	 		\dot V &=& \alpha e_r \dot e_r + \sigma\left( \frac{k_1}{U_2} e_\psi + \frac{k_2 - 1}{U_2} e_r  \right) \left[ \frac{k_1}{U_2} \dot e_\psi +  \frac{k_2 - 1}{U_2} \dot e_r \right] , \\
	 		&=& \alpha e_r \left( - e_r - U_2 \sigma (z + f_1)  + f_2 \right) + \sigma (z) \left[ \frac{k_1}{U_2} e_r + \frac{k_2 - 1}{ U_2} \left( - e_r - U_2 \sigma( z + f_1) + f_2\right) \right].\\
	 \eeqnum
Using the inequality, $|ab| \le \frac{a^2 + b^2}{2}$, and taking $\alpha U_2 = \frac{k_1 - k_2 + 1}{U_2} > 0$, we get
	 \beq
	 		\dot V &\le& - \frac{\alpha}{2} e_r^2 - \left(k_2 - 1 \right) \sigma \left( z \right) \sigma \left(z + f_1 \right)
	 		+ \alpha U_2 e_r \left(  \sigma \left( z \right) - \sigma \left(z + f_1 \right) \right)
	 		+\frac{\alpha}{2} f_2^2  + \frac{k_2 - 1}{U_2} \left| f_2 \right| .
	 \eeq
Using the inequalities $\left| \sigma \left( z \right) - \sigma \left(z + f_1 \right) \right|\leq  \left| f_1 \right|$ and $\sigma \left( z \right) \sigma \left(z + f_1 \right)\geq \sigma ^ 2 \left( z \right) -\left| f_1 \right |$, one has 
$$
\dot V \le - \frac{\alpha}{2} e_r^2 - \left(k_2 - 1 \right) \sigma^2 \left( z \right) +  \left(k_2 - 1 \right)\left| f_1 \right| + \alpha U_2 \left| e_r \right| \left| f_1 \right|
	 		+\frac{\alpha}{2} f_2^2  + \frac{k_2 - 1}{U_2} \left| f_2 \right|.
$$
After a sufficiently large time interval $T$, it is assured that $|U_2 f_1| < \frac1{6}$ $\forall t > T$, and
	 \beqnum
	 		\dot V &\le& -\frac{\alpha}{3} e_r^2 - (k_2 - 1) \sigma^2 \left( z \right) + O\left(f_1^2,\left|f_2\right|, f_2^2  \right).
	 \eeqnum
From here, we obtain that 
$\mathop {\lim \sup }\limits_{t \to \infty } \left| e_r \right| = \mathop {\lim \sup }\limits_{t \to \infty } \left| \sigma(z) \right| = 0$.
\end{proof}
Following the same steps as used in the previous section, we now demonstrate the convergence of the error variables ($e_u$,$e_v$,$e_x$,$e_y$) of System $(S_e)$.\begin{lemme}\label{conv_eu_ev_o}
Consider the dynamics of $e_u$ and $e_v$ presented in Equation \eqref{sys_error}. Then, the control
\beqnum\label{control_w1 + f}
w_1 = -U_1 \sigma \left(\frac{\xi \hat e_u}{U_1}\right) - \rho \sigma_1 (.)
\eeqnum
ensures that $e_u$ and $e_v$ are bounded and again verifiy the estimate \eqref{toto_w1}.
\end{lemme}
\begin{proof}
The argument exactly follows the line of the argument of Lemma \ref{conv_eu_ev} with 
the controller $w_1$ written as $w_1 := -U_1 \sigma \left(\frac{\xi e_u}{U_1} + f\right) - \rho \sigma_1 (.)$, where $f:= -\frac{\xi}{U_1} f_u$ converges to zero asymptotically and is integrable over $\mathbb{R}_+$, and by using  the inequality $ x \sigma(x + f) \geq x \sigma(x) - | x \sigma(f)|$.
\end{proof}
Similarly, the proof of convergence of $W_1$ to zero and $W_2$ to a finite limit, presented in Lemma \ref{lemmeW}, holds true and one concludes by showing that the limit of $W_2$  is zero as well.
\end{proof}
\section{Simulations}
The performance of the presented controller is illustrated by simulation. We apply the controller on an example of a monohull vessel, as considered in \cite{Do2002}. The length of this vessel is $32$ m, and a mass of $118 \times 10^3$ kg.
The parameters of the damping matrices as given as follow:
\beqnum
\begin{array}{ccccccc}
	d_1 &=& 215 \times 10^2 Kgs^{-1},\quad d_2 &=& 97 \times 10^3 Kgs^{-1},\quad d_3 &=& 802 \times 10^4 Kg m^2 s^{-1},\\
	m_1 &=& 120 \times 10^3 Kg,\quad m_2 &=& 172.9 \times 10^3 Kg,\quad m_3 &=& 636 \times 10^5 Kg m^2.\\
\end{array}
\eeqnum

Based on these physical parameters, we find the parameters of System \eqref{system} as
\beqnum
a= 0.179,\ b = 0.561,\ c = 0.694, \ \beta = 0.126 , \kappa = 8.32 \times 10^{-4}.
\eeqnum

Then, the parameters of the controller and the normalized system $(S)$ are given by
\beq
	a_1 = 1.421, \ b_1 = 4.449, \  d = 0.126,\
	k_1 = 10, \ k_2 = 10,\ U_2 = 0.1, 
	\ U_1 = \frac{a_1}{2}, \ \rho = \frac{a_1}{4}, \ M = 0.1 .
\eeq
The reference trajectory is generated by considering the surge force and the yaw moment as constants $\tau_{1,re} = 10$ and $\tau_{2,re} = 0.05$ with the initial values
\beq
	\left(x_{re}(0),y_{re}(0),\psi_{re}(0),u_{re}(0),v_{re}(0),r_{re}(0)\right) = \left(0 \ m,0 \ m,0 \ rad,\ 0 ms^{-1},0 \ ms^{-1},0 \  rads^{-1}\right)).
\eeq
The initial conditions of the vessel are as follows:
\beq
	\left(x(0),y_{re}(0),\psi(0),u(0),v(0),r(0)\right) = \left(50 \ m,-150 \ m,\frac{\pi}{4} \ rad,\ 50 ms^{-1},0 \ ms^{-1},0 \  rads^{-1}\right).
\eeq
The reference trajectory and the vessel are shown in a 2D coordinate plane in Figure \ref{plan}.The vessel converges to the reference trajectory asymptotically and similarly for the position errors graph in Figure \ref{exy}. The orientation error and its derivative also converge to zero, as seen in Figure \ref{epsir}. The convergence of $e_u$ and $e_v$ is shown in Figure \ref{euv}. Figures \ref{tau1} and \ref{tau2} show the control signals $\tau_1$ and $\tau_2$ respectively and the controllers are clearly bounded. This is an essential property in real systems, where the control signals are constrained.
\section{Conclusion}
In this paper, we have addressed the problem of tracking of an underactuated surface vessel with only surge force an yaw moment. The proposed controller ensures global asymptotic tracking of the vessel, following a reference trajectory modeled by a virtual vessel. It is also shown that the stability of the system is not affected if the state measurements are replaced by observers. The using of saturated inputs is essential as in real life the actuators have limited output. Simulation results illustrate the performance of the controller.

\begin{figure}[h!]
 \begin{minipage}[b]{0.5\linewidth}
  \centering\epsfig{figure=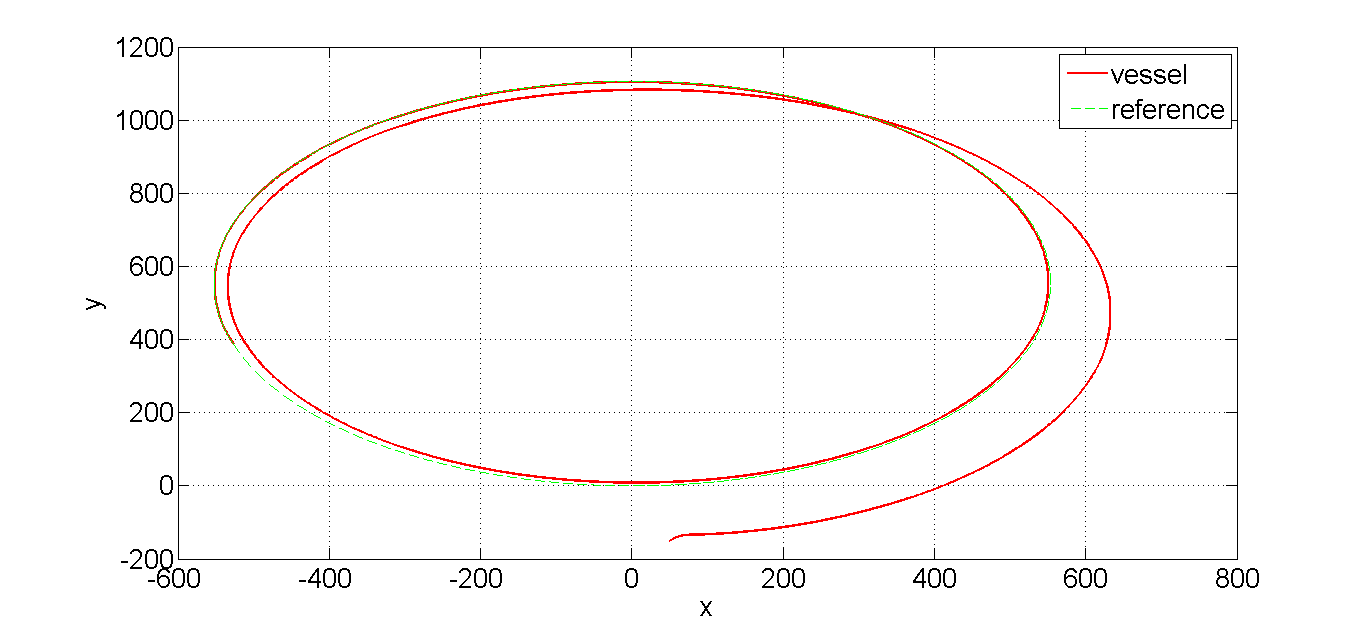,width=\linewidth}
  \caption{Reference trajectory and the vessel  \label{plan}}
 \end{minipage} \hfill
 \begin{minipage}[b]{0.5\linewidth}
  \centering\epsfig{figure=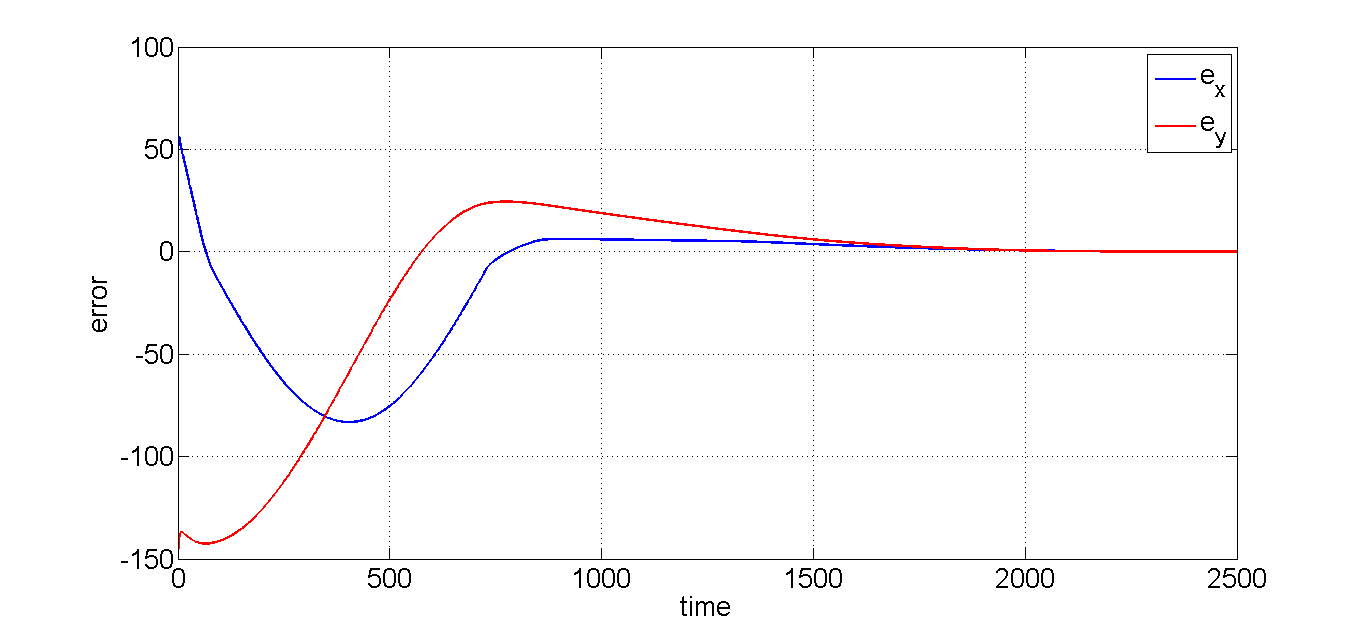,width=\linewidth}
  \caption{Errors $e_x$ and $e_y$  \label{exy}}
 \end{minipage}
\end{figure}

\begin{figure}[h!]
 \begin{minipage}[b]{.5\linewidth}
  \centering\epsfig{figure=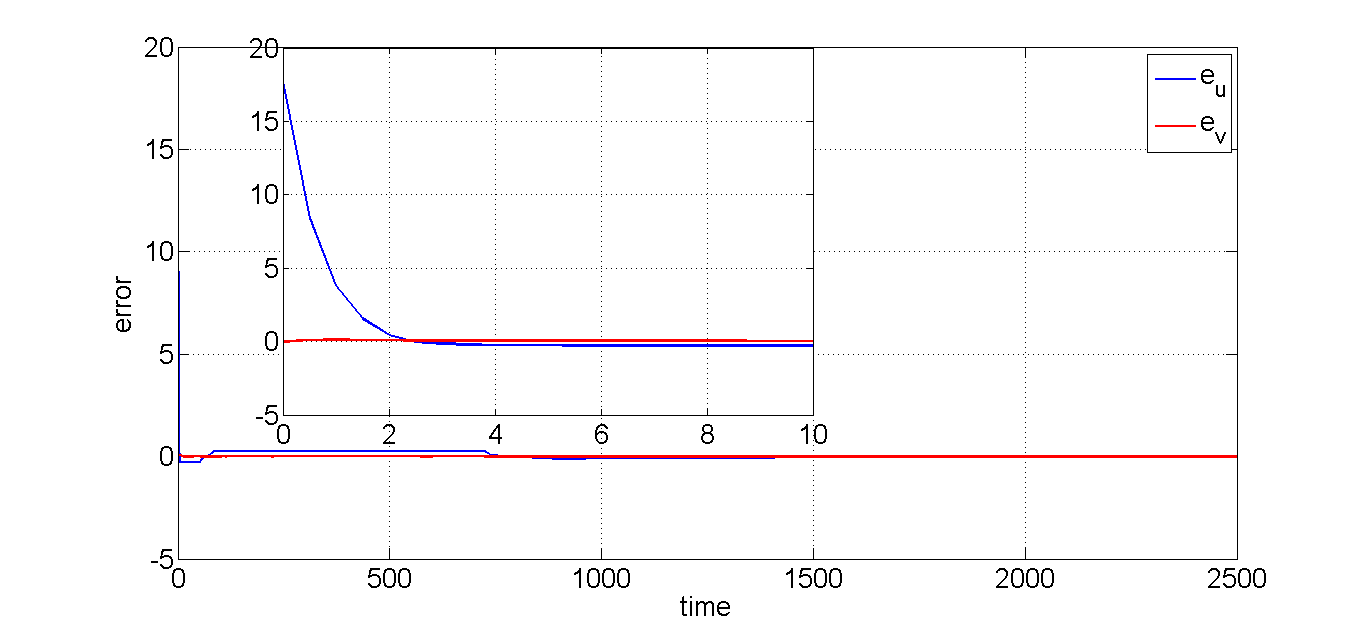,width=\linewidth}
  \caption{Errors $e_u$ and $e_v$  \label{euv}}
 \end{minipage} \hfill
 \begin{minipage}[b]{.5\linewidth}
  \centering\epsfig{figure=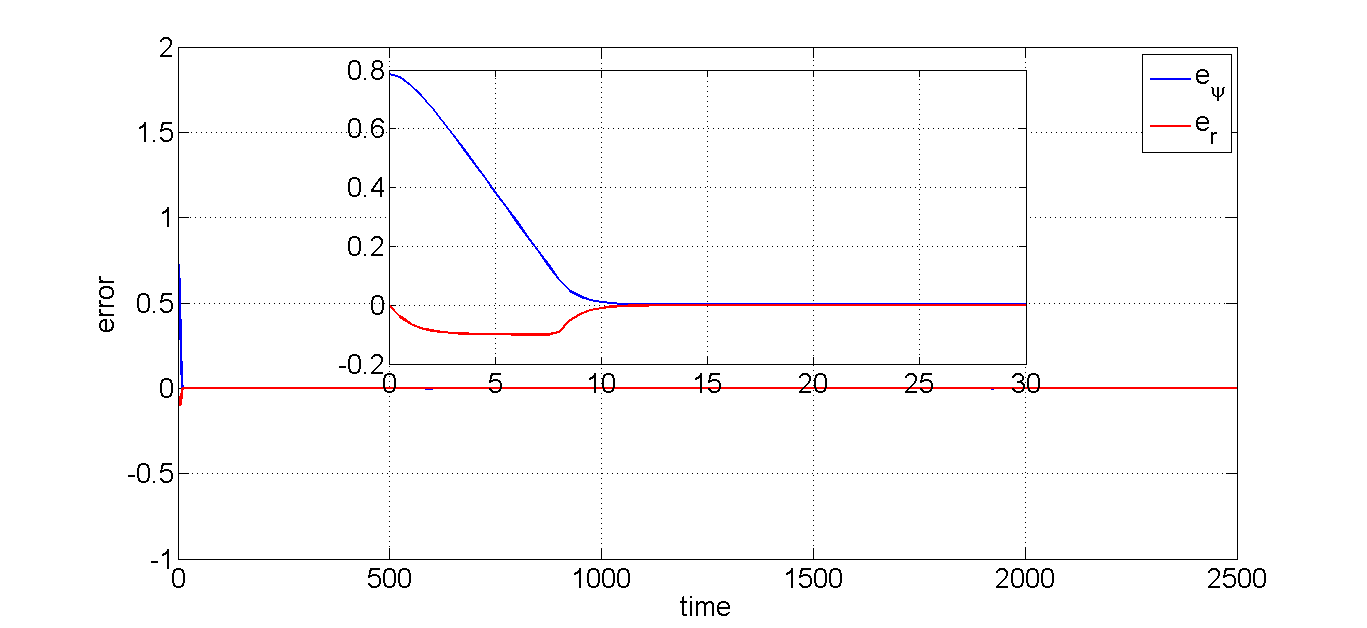,width=\linewidth}
  \caption{Errors $e_\psi$ and $e_r$ \label{epsir}}
 \end{minipage}
\end{figure}

\begin{figure}[h!]
 \begin{minipage}[b]{.5\linewidth}
  \centering\epsfig{figure=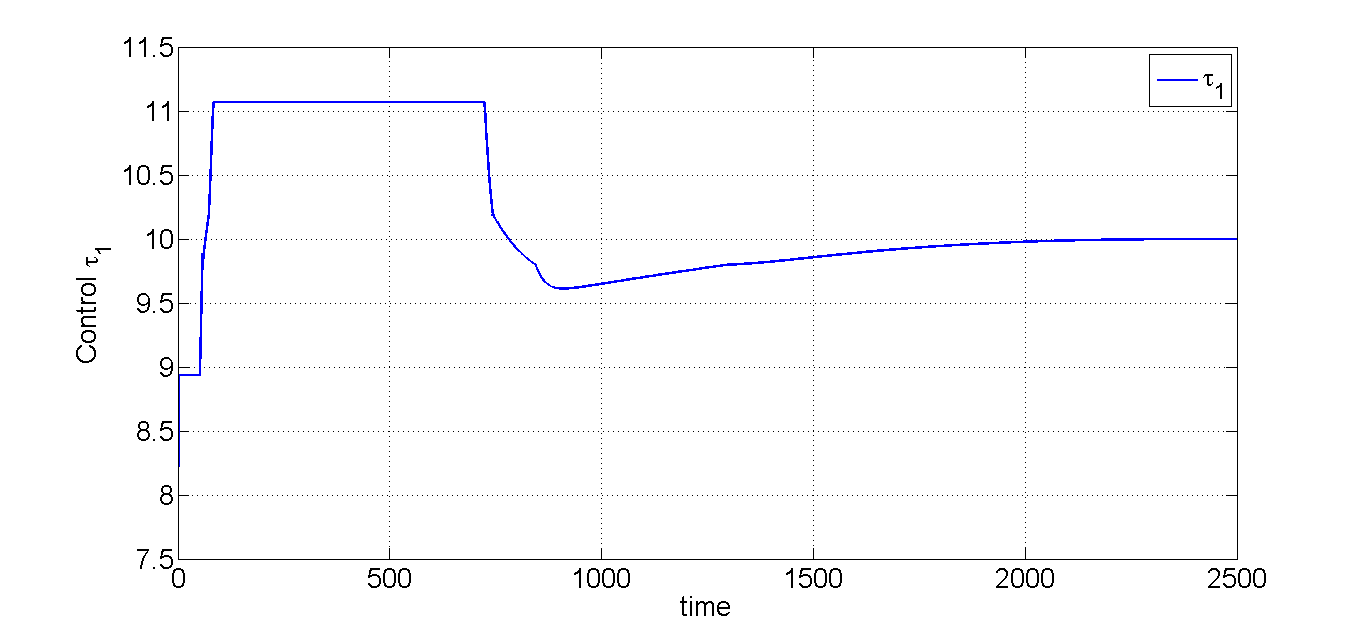,width=\linewidth}
  \caption{Control $\tau_1$ \label{tau1}}
 \end{minipage} \hfill
 \begin{minipage}[b]{.5\linewidth}
  \centering\epsfig{figure=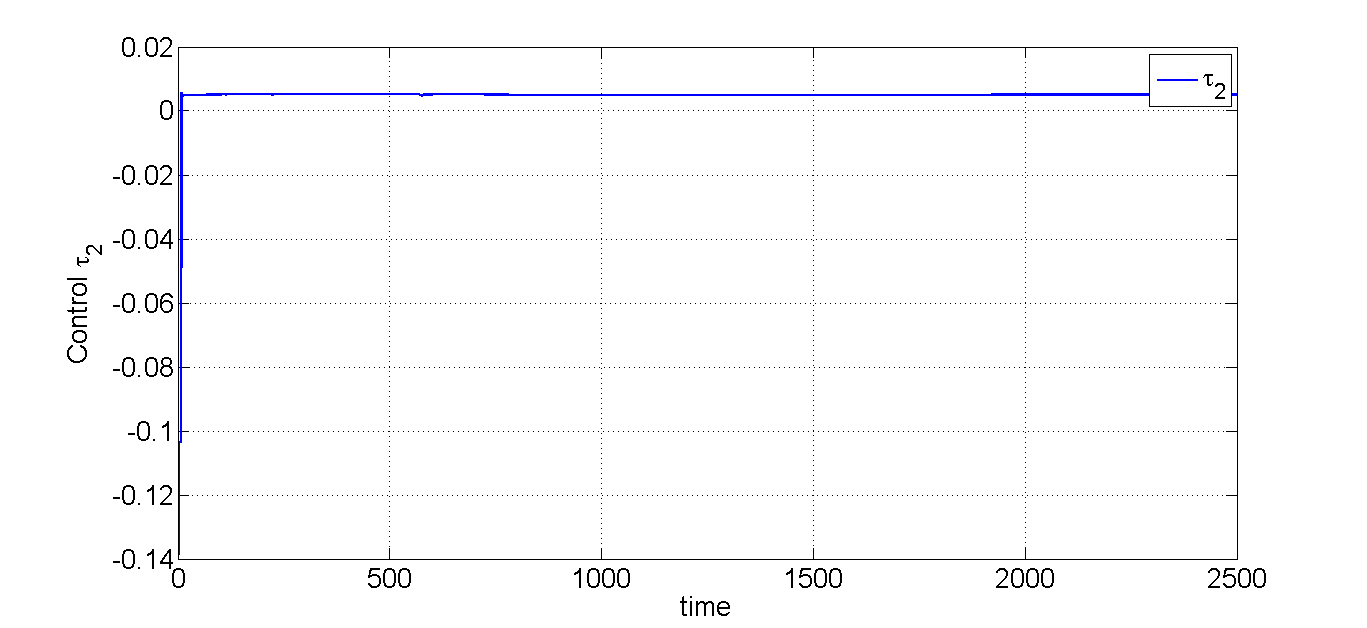,width=\linewidth}
  \caption{Control $\tau_2$ \label{tau2}}
 \end{minipage}
\end{figure}


\pagebreak
\bibliography{Bateau_bib}

\end{document}